\documentclass[11pt]{amsart}

\usepackage{amssymb}
\usepackage{amsmath}
\usepackage{amsfonts}
\usepackage{graphicx}
\usepackage{graphics}
\usepackage{amsbsy}
\usepackage{amscd}

\usepackage{color}

\usepackage[all,knot,poly]{xy}
\usepackage{ifpdf}
\usepackage{ucs}
\usepackage[utf8x]{inputenc}
\usepackage[T1]{fontenc}
\usepackage{latexsym}
\usepackage{array}
\usepackage{enumerate}

\def\gd{\mathrm{gd}}
\def\imm{\mathrm{imm}}
\def\stablespan{\mathrm{stablespan}}
\def\span{\mathrm{span}}
\def\cat{\mathrm{cat}}
\def\TC{\mathrm{TC}}
\def\n{\overline{n}}
\def\m{\overline{m}}
\def\x{\overline{x}}
\def\P{\mathrm{P}}
\def\Sq{\mathrm{Sq}}

\def\invariant{$\mathbb{Z}_t$-invariant }
\def\L{\mathbb{C}\P}
\def\cpnt{\L_{\overline{n}}(t)}
\def\cpmt{\L_{\overline{m}}(t)}

\newtheorem{theorem}{Theorem}[section]
\newtheorem{proposition}[theorem]{Proposition}
\newtheorem{lema}[theorem]{Lemma}
\newtheorem{definition}[theorem]{Definition}
\newtheorem{corolario}[theorem]{Corollary}
\newtheorem{nota}[theorem]{Remark}

\newtheorem{ejemplo}[theorem]{Example}
\newtheorem{examples}[theorem]{Examples }

\title{Complex-Projective and lens Product Spaces}

\author{Jes\'us Gonz\'alez and Maurilio Velasco}
\date{\today}

\begin{document}

\begin{abstract}
Let $t$ be a positive integer. Following work of D.~M.~Davis, we study the topology of complex-projective product spaces, i.e.~quotients of cartesian products of odd dimensional spheres by the diagonal $S^1$-action, and of the $t$-torsion lens product spaces, i.e.~the corresponding quotients when the action is restricted to the $t^{\mathrm{th}}$ roots of unity. For a commutative complex-oriented cohomology theory $h^*$, we determine the  $h^*$-cohomology ring of these spaces (in terms of the $t$-series for $h^*$, in the case of $t$-torsion lens product spaces). When $h^*$ is singular cohomology with mod~2 coefficients, we also determine the action of the Steenrod algebra. We show that these spaces break apart after a suspension as a wedge of desuspensions of usual stunted complex projective ($t$-torsion lens) spaces. We estimate the category and topological complexity of complex-projective and lens product spaces, showing in particular that these invariants are usually much lower than predicted by the usual dimensional bounds. We extend Davis' analysis of manifold properties such as immersion dimension, (stable-)span, and (stable-)parallelizability of real projective product spaces to the complex-projective and lens product cases.
\end{abstract}

\maketitle \tableofcontents

\section{Preliminaries}\label{Int}
The family of real and complex projective spaces provide a standard benchmark for testing questions and finding interesting phenomena in algebraic topology. This basic family of manifolds has now been extended in~\cite{DDa} to that of projective \emph{product} spaces. Yet, Davis considered only the real case. This paper should be thought of as a companion of~\cite{DDa} in that we extend Davis' viewpoint to the complex case---therefore completing the extended benchmark. In fact, the complex analysis is streamlined (and complemented) due, in part, to a unifying ``lens'' approach.

\medskip
For a positive integer $t$ and an $r$-tuple $\n=(n_1,\dots, n_r)$ of non-negative integers define
$$L_{\n}(t)=S^{2n_1+1}\times\cdots\times
S^{n_r+1}\left/\,(x_1,\dots,x_r)\thicksim(\lambda x_1,\dots,\lambda x_r)\right.$$ 
where $x_i \in S^{2n_i+1}$ and $\lambda \in\mathbb{Z}_t$. Here we think of $\mathbb{Z}_t$ as the $t^{\mathrm{th}}$ roots of unity in the unit circle $S^1 \subseteq \mathbb{C}$. This is a manifold of dimension $2(n_1+\cdots+ n_r)+r$ which we call a $t$-torsion lens product space. We might occasionally simplify the notation to $L_{\overline{n}}$ when the torsion $t$ is understood from the context. We also consider the case $t=\infty$, for which we agree to interpret $\mathbb{Z}_\infty$ as $S^1$. The resulting $(2(n_1+\cdots+n_r)+r-1)$-dimensional manifold is called a complex-projective product space, and is denoted by $\mathbb{C}\P_{\n}$. 

\medskip
Unless explicitly noted otherwise, throughout the paper we will assume $n_1\leq n_2\leq\cdots\leq n_r$. Also, for the sake of simplifying some statements, it will be convenient to unify notation by writing 
$$\mathbb{C}\P_{\n}(t)=\begin{cases}
\mathbb{C}\P_{\n}, & t=\infty.\\
L_{\n}(t), & \mbox{$t$ a positive integer.}
\end{cases}$$

\medskip
With $r=1$ these constructions yield standard complex projective spaces $\L^{n_1}=\L_{(n_1)}(\infty)$ and lens spaces $L^{2n_1+1}(t)=L_{(n_1)}(t)$ for finite $t$. Additionally, the special case with $t=2$ recovers those projective product spaces constructed in~\cite{DDa} from products of odd dimensional spheres. 

\medskip
For $0<t\leq\infty$ there is a complex line bundle $\gamma_{\n}(t)$ over $\L_{\n}(t)$ whose $k$-fold Whitney sum $k\,\gamma_{\n}(t)$ has total space
\begin{equation}\label{totalspace}
S^{2n_1+1}\times\cdots\times S^{2n_r+1}\times\mathbb{C}^k\left/\,(\x,z)\thicksim(\lambda\x,\lambda z)\right.
\end{equation}
where $\x \in S^{2n_1+1}\times\cdots\times S^{2n_r+1}$, $z\in \mathbb{C}^k$, and $\lambda \in \mathbb{Z}_{t}$. In particular
\begin{equation}\label{spherebundle}
S((k+1)\,\gamma_{\n}(t))= \L_{(n_1,\dots,n_r,k)}(t).
\end{equation}
Note that for $t'=\infty$, or else for $t$ dividing $t'$, the canonical projection $\pi=\pi_{t,t'}\colon\L_{\n}(t)\to\L_{\n}(t')$ satisfies 
\begin{equation}\label{pullbackttprima}
\pi^*(\gamma_{\n}(t'))=\gamma_{\n}(t).
\end{equation}
Likewise, if $\overline{m}$ is obtained from $\overline{n}$ by removal of some coordinates, then there is an obvious projection $p\colon\L_{\overline{n}}(t)\to\L_{\overline{m}}(t)$ satisfying 
\begin{equation}\label{pullbackttprimaprima}
p^*(\gamma_{\m}(t))=\gamma_{\n}(t).
\end{equation}
Note that if $\overline{m}=(n_{t_1},\ldots,n_{t_s})$ with $t_1=1$, then $p$ is a retraction with respect to the inclusion $j\colon\L_{\overline{m}}(t)\to\L_{\overline{n}}(t)$ induced by
$$
(x_{t_1},\ldots,x_{t_s})\mapsto(x_{t_1},\ldots,x_{t_1},x_{t_2},\ldots,x_{t_2},\ldots,x_{t_s},\ldots,x_{t_s})
$$
where each $x_{t_j}$ is repeated as needed so to account for the coordinates that have been removed from $\overline{n}$ in forming $\overline{m}$. In particular,~(\ref{pullbackttprimaprima}) yields
\begin{equation}\label{pullbackttprimaprimaprima}
j^*(\gamma_{\n}(t))=\gamma_{\m}(t).
\end{equation}
In view of~(\ref{pullbackttprima}) and~(\ref{pullbackttprimaprima}), the bundles $\gamma_{\n}(t)$ generalize the \emph{dual} Hopf bundles over usual complex projective spaces. Indeed, for $k=r=1$ and $t=\infty$, (\ref{totalspace})~describes the complex conjugate of the Borel construction for the canonical Hopf bundle over $\L^{n_1}$. In any case, for the purposes of the discussion in Section~\ref{mnfdppts}, it will be convenient to keep in mind that the realification $r\left(\gamma_{(n_1)}(t)\right)$ agrees with the realification of the corresponding Hopf bundle.

\section{Cohomology of complex-projective and lens product spaces}\label{Cohmly}

Let  $h$ be a commutative complex-oriented cohomology theory. Recall our assumption $n_1\leq\cdots\leq n_r$.
\begin{theorem}\label{additive}
For any $0<t\leq\infty$, there is an isomorphism of graded $h^*$-modules
\begin{equation}\label{ladescrip}
h^*(\L_{\n}(t))\cong h^*(\L_{(n_1)}(t))\otimes\Lambda_{h^*}[x_2,\dots,
x_r]
\end{equation} 
where $\dim(x_i) = 2n_i+1\hspace{.1mm}$ for $\hspace{.1mm}i=2,\dots r$. 
\end{theorem}

It will be clear from the proof below, that the tensor-product description in~(\ref{ladescrip}) is functorial with respect to the maps $\pi_{t,t'}$ at the end of the previous section. (This is why we can safely remove the parameter $t$ from the notation for the classes $x_i$.) For instance, the various $x_i$ arise as Thom classes, which behave nicely due to~(\ref{pullbackttprima}). Likewise, a nice functorial behavior will also be clear for the maps $p$ at the end of the previous section provided the first coordinate $n_1$ of $\overline{n}$ is kept in $\overline{m}$. In such a case, the induced morphism $p^*$ is injective with image the $h^*(\L_{(n_1)}(t))$-submodule generated by the products of the classes $x_i$ for which $n_i$ is one of the coordinates of $\overline{n}$ that are kept in forming $\overline{m}$. The latter observation will be crucial in the proof of the  splitting of $\Sigma\L_{\overline{n}}(t)$ (Corollary~\ref{decomthom} below).

\begin{proof}[Proof of Theorem~$\ref{additive}$] The  case $r=1$ is vacuously true, so assume the result is known for $\L_{\m}(t)$ with $\m=(n_1,\dots,n_{r-1})$. Since $\L_{\n}(t)=S((n_r+1)\gamma_{\m}(t))$, there is a cofibration
\begin{equation}\label{Cof}
\L_{\n}(t)\stackrel{p}{\longrightarrow}
\L_{\m}(t)\stackrel{i}{\longrightarrow} T
\end{equation} where $T$ stands for the
Thom space $T\left((n_r+1)\gamma_{\m}(t)\right)$. As noted at the end of the previous section, the map $p$ is a retraction. Consequently, the long exact sequence in $h^*$-cohomology for~(\ref{Cof}) breaks into a family of split short exact sequences
\begin{equation}\label{splitses}
0\longrightarrow h^*(\L_{\m}(t))\stackrel{p^*}{\longrightarrow}
h^*(\L_{\n}(t))\stackrel{\delta}{\longrightarrow}
h^{*+1}(T)\longrightarrow 0.
\end{equation}
Thus, as a graded $h^*$-module,
\begin{equation}\label{hmod}
h^*(\L_{\n}(t))\cong h^*(\L_{\m}(t))\oplus h^{*+1}(T).
\end{equation}
Let $x_r\in h^{2n_r+1}(\L_{\n}(t))$ correspond to the Thom class $U_{(n_r+1)\gamma_{\m}(t)}$ in $h^{2(n_r+1)}(T)$, so that~(\ref{hmod}) takes the form $h^*(\L_{\n}(t))\cong h^*(\L_{\m}(t))\otimes\Lambda_{h^*}[x_r]$. The result follows by induction.
\end{proof}
Note that the proof above implies that, in terms of the isomorphism in Theorem~\ref{additive}, the map $h_1^*(\L_{\n}(t))\to h_2^*(\L_{\n}(t))$ induced by a morphism $\alpha\colon h^*_1\to h^*_2$ of commutative complex-oriented cohomology theories takes the form $\alpha_{L_{\n}}=\alpha_{L_{(n_1)}}\otimes1$.

\medskip
Since each projection map $p$ in~(\ref{Cof}) induces a monomorphism of $h^*$-algebras, a full description of the ring structure of $h^*(\L_{\n}(t))$ requires (a) an explicit description of the ring $h^*(\L_{(n_1)}(t))$, and (b) a determination of the value of each square $x_i^2$ in Theorem~$\ref{additive}$. After reviewing the well-known solution for the first task (Examples~\ref{ringmod2} and~\ref{ringZ}), we will focus on the second task (Theorem~\ref{productsZ}).

\begin{ejemplo}\label{ringmod2}{\em
Let $z\in h^2(\L^\infty)$ be the orientation of $h^*$. It is standard that, as an $h^*$-algebra, $h^*(\L^{n_1})$ is generated by the restriction of $z$ (denoted again by $z$) subject to the single relation $z^{n_1+1}=0$. Further, if the positive integer $t$ is a not a zero divisor in $h^*$ then, as an $h^*$-algebra,
\begin{equation}\label{zerodivisor}
h^*(L^{2n_1+1}(t))\cong h^*[z]\left/\left(z^{n_1+1},[t](z)\right)\right.\oplus h^*\cdot\omega
\end{equation}
where $\dim(\omega)=2n_1+1$ and $[t]$ stands for the $t$-series associated to $h^*$. Here $z$ is the image of the corresponding class $z$ in $\L^{n_1}$ under the canonical projection $L^{2n_1+1}(t)\to\L^{n_1}$. Furthermore the class $\omega$ pulls back under the canonical projection $\mathrm{proj}\colon S^{2n_1+1}\to L^{2n_1+1}(t)$ to the $t$-multiple of the generator $\iota$ corresponding (under the suspension isomorphism) to $1\in h^*=\overline{h}^{\hspace{.3mm}*}\hspace{-.6mm}(S^0)$. In particular, it follows from the commutative diagram
$$\xymatrix{
S^{2n_1+1}\times\cdots\times S^{2n_r+1}\ar[d]_{\pi_1}\ar[r]^{\mathrm{\ \hspace{15mm} proj}}&L_{\n}(t)\ar[d]^{p}\\S^{2n_1+1}\ar[r]_{\mathrm{\ \hspace{-4mm} proj}}&L^{2n_1+1}(t)
}$$
that, as an element in $h^*(L_{\n}(t))$, $\omega$ pulls back under the top horizontal map to the $t$-multiple of $\iota\otimes1\otimes\cdots\otimes1$.
}\end{ejemplo}

\begin{ejemplo}\label{ringZ}{\em
Let $L=L^{2n_1+1}(2^e)$. Then $H^*(L;\mathbb{Z}_2)$ is the $\mathbb{Z}_2$-algebra generated by a class $y$ of dimension $1$, and the mod 2 reduction of $z\in H^*(L;\mathbb{Z})$ (again denoted by $z$), subject only to the relations $y^2=2^{e-1} z$ and $z^{n_1+1}=0$. These generators are connected through the $e$-th Bockstein morphism.
}\end{ejemplo}

As for the ring structure in $h^*(\L_{\n}(t))$, we have:

\begin{theorem}\label{productsZ}
If $2$ is a not a zero divisor in $h^*$, then the relations $x_i^2=0$ hold in Theorem~$\ref{additive}$. For $h^*=H^*(-\hspace{.5mm};\mathbb{Z}_2)$ we actually have that the Steenrod squares act as $\hspace{.2mm}\Sq(x_i)=(1+z)^{n_i+1}x_i$.
\end{theorem}

In this result, the second assertion implies that the conclusion of the first assertion also holds for $h^*=H^*(-\hspace{.5mm};\mathbb{Z}_2)$---even though $2=0$ in $\mathbb{Z}_2$. In turn, the second assertion can be proved just as in~\cite{DDa}---noticing from Example~\ref{ringZ} and the functoriality discussion after the statement of Theorem~\ref{additive}
that it suffices to consider the case $t=\infty$: As noted in~(\ref{pullbackttprimaprima}), the iterated projection map $p\colon\L_{\n}(\infty)\to\L_{(n_1)}(\infty)$ pulls the dual Hopf bundle over $\L_{(n_1)}(\infty)$ back to $\gamma_{\n}(\infty)$. Therefore, in terms of the notation in~Example~\ref{ringmod2}, the total Stiefel-Whitney class of $\gamma_{\n}(\infty)$ is $W(\gamma_{\n}(\infty))=1+z$. Thus, in the proof of Theorem~\ref{additive} (with $h^*=H^*(-\hspace{.5mm};\mathbb{Z}_2)$) we have
\begin{eqnarray*}
\Sq(U_{(n_r+1)\gamma_{\m}(\infty)})&=&W((n_r+1)\gamma_{\m}(\infty)) U_{(n_r+1)\gamma_{\m}(\infty)}\\
&=&(1+z)^{n_r+1}U_{(n_r+1)\gamma_{\m}(\infty)},
\end{eqnarray*}
and the second assertion in Theorem~\ref{productsZ} follows by induction, and from the fact that the maps in~(\ref{splitses}), including the section, preserve the action of the Steenrod algebra.

\medskip
The proof of the case $h^*=H^*(-\hspace{.5mm};\mathbb{Z})$ (and $t=2$) of Theorem~\ref{productsZ} given in~\cite{DDa} is rather circuitous, as it depends on a preliminary knowledge of Theorem~\ref{productsZ} for $h^*=H^*(-\hspace{.5mm};\mathbb{Z}_2)$ and, more importantly, on an (any way interesting) stable decomposition of (real) product projective spaces. Indeed, the stable decomposition forces the behavior of differentials in the Cartan-Leray spectral sequence for the action of $\mathbb{Z}_2$ on the product of spheres defining the projective space under consideration. As a much more direct alternative, our proof of the (more general) Theorem~\ref{productsZ} is conceptually simpler, taking key advantage of this paper's ``lentified'' viewpoint.

\begin{proof}[Proof of Theorem~$\ref{productsZ}$]
It only remains to prove the first assertion. As observed in the paragraph following the statement of the theorem, it suffices to consider the situation for $t=\infty$. In that case, the desired assertion follows from the graded commutativity of $h^*$, the fact that each class $x_i$ is odd dimensional (these two conditions imply $2x_i^2=0$), and the fact that $h^*(\mathbb{C}\P_{\n}(\infty))$ is 2-torsion free---because $h^*$ is so by hypothesis.
\end{proof}

\section{Stable and unstable decompositions}

Definition~\ref{invariant} below isolates the main ingredient in~\cite{DDa} leading to an unstable decomposition of (real) projective product spaces. The proof of~\cite[Theorem~2.20(1)]{DDa} then generalizes directly to give Proposition~\ref{cartesian_product} below.

\begin{definition}\label{invariant}
Let $0<t\leq\infty$. A map $f:S^{2a+1}\times S^{2b+1}\to S^{2b+1}$ is called \invariant provided:
\begin{itemize}
\item[$(i)$] $f(\lambda z,\lambda \omega)=f(z,\omega)\,$ for $z\in S^{2a+1}$, $\omega\in S^{2b+1}$, and $\lambda\in\mathbb{Z}_t\subseteq S^1$.
\item[$(ii)$] $f(z,-):S^{2b+1}\to S^{2b+1}$ is an homeomorphism for any $z\in S^{2a+1}$.
\end{itemize}
\end{definition}

\begin{nota}\label{ccoojj}{\em
In terms of the map $c\colon S^{2a+1}\to S^{2a+1}$ that conjugates all complex coordinates, condition $(i)$ means that the composite $g=f\circ(c\times 1)$ satisfies 
\begin{itemize}
\item[$(i')$] $g(\lambda^{-1}z,\lambda\omega)=g(z,\omega)$ for $z\in S^{2a+1}$, $\omega\in S^{2b+1}$, and $\lambda\in\mathbb{Z}_t\subseteq S^1$.
\end{itemize}
}\end{nota}

\begin{proposition}\label{cartesian_product}
Let $0<t\leq\infty$. There are homeomorphisms $$\cpnt\cong\cpmt\times\prod_{i\in I_t}S^{2n_i+1}$$ where $I_t$ consists of the indices $i\in\{2,\ldots,r\}$ for which there is a \invariant map $f_i:S^{2n_1+1}\times S^{2n_i+1}\to S^{2n_i+1}$, and $\overline{m}$ is the tuple obtained from $\overline{n}$ by removing those $n_i$ with $i\in I_t$.
\end{proposition}

$\mathbb{Z}_2$-invariant maps can be obtained as restrictions of normed ($\mathbb{R}$-linear) multiplications. In fact, as observed in the proof of~\cite[Theorem~2.20]{DDa}, a $\mathbb{Z}_2$-invariant map $f_i\colon S^{2n_1+1}\times S^{2n_i+1}\to S^{2n_i+1}$ can exist only when the sphere $S^{2n_i+1}$ splits off as a cartesian factor of $\L_{\overline{n}}$ which, in turn, happens if and only if $\mathbb{R}^{2n_i+2}$ admits a Cliff$(2n_1+1)$-module structure---a condition which is classical and well understood. But the existence of \invariant maps for $t>2$ does not appear to have been considered in the literature. In the simplest case, Remark~\ref{ccoojj} shows that the complex vector space structure on $\mathbb{C}^{n_i+1}$ gives rise to a \invariant map $S^1\times S^{2n_i+1}\to S^{2n_i+1}$ for any~$t$. In particular:

\begin{corolario}\label{casoelemental}
If $n_1=0$ and $0<t\leq\infty$, then
$$\cpnt\cong\begin{cases}
\prod_{2\leq i\leq r}S^{2n_i+1},&t=\infty;\\
\prod_{1\leq i\leq r}S^{2n_i+1},&\mbox{otherwise.}
\end{cases}$$
\end{corolario}

\begin{ejemplo}\label{paran10}{\em
Since products of spheres are stably parallelizable, and in fact parallelizable provided there are at least two factors one of which is odd dimensional (see~\cite{MR0200937} for instance), we immediately see that, except for $\L_{(0,n_2)}(\infty)\approx S^{2n_2+1}$ with $n_2\not\in\{0,1,3\}$, any $\L_{\n}(t)$ with $n_1=0$ is parallelizable.
}\end{ejemplo}

\begin{proposition}\label{normed2} For each integer $k\geq1$ there is a map $\mu_k:\mathbb{C}^2\times \mathbb{C}^{2k}\to
\mathbb{C}^{2k}$ which is normed $($i.e.~so that $|\mu(z,w)|=|z|\cdot|w|)$ and satisfies
\begin{itemize}
\item condition~$(i)$ for $z\in\mathbb{C}^2$, $w\in\mathbb{C}^{2k}$, and $\lambda\in S^1$, and
\item condition~$(ii)$ for $z\in S^3$
\end{itemize}
in Definition~$\ref{invariant}$.
\end{proposition}

\begin{proof} The argument is a partial (unitary) analogue of the well-known meth\-od to construct orthogonal multiplications from explicit solutions to the Hurwitz-Radon (orthogonal) matrix equations. Consider the map 
$\delta:\mathbb{R}^4\times \mathbb{C}^{2}\to
\mathbb{C}^{2}$ given by $$\delta(t_1,t_2,t_3,t_4,w_1,w_2)=
t_1A_1w+t_2A_2w+t_3A_3w+t_4I_2w$$
where $I_2$ is the $2\times2$ unit matrix, $w$ is the transpose of the row vector $(w_1,w_2)$, and
\begin{equation}\label{tresmatrices}
A_1 = \left( \begin{array}{rr} i & 0  \\
0 & -i \end{array} \right),\quad A_2 = \left(
\begin{array}{rr} 0 & 1 \\ -1 & 0 \end{array}
\right), \mbox{ \ \ and}\quad A_3 = \left(
\begin{array}{rr} 0 & i \\ i & 0  \end{array}
\right).
\end{equation}
Identify $\mathbb{R}^4=\mathbb{C}^2$ via $z_1=t_1+it_{4}$ and $z_2=t_2+it_{3}$, and precompose $\delta$ with the map $\gamma:\mathbb{C}^2\times\mathbb{C}^{2}\to\mathbb{C}^{2}\times
\mathbb{C}^{2}$ given
by $\gamma(z_1,z_2,w_1,w_2)=(z_1,z_2,w_1,\overline{w}_2)$ to get a map $\mu_1:\mathbb{C}^2\times\mathbb{C}^2\to\mathbb{C}^2$. Explicitly
$$\mu_1(z_1,z_2,w_1,w_2)=
(i\overline{z}_1w_1+z_2\overline{w}_2,-\overline{z}_2w_1-iz_1\overline{w}_2)$$
which, by direct verification, satisfies the required conditions for $k=1$. More generally, the required map $\mu_k:\mathbb{C}^2\times\mathbb{C}^{2k}\to\mathbb{C}^{2k}$ is given by
$$
\mu_k(\alpha,\beta_1,\ldots,\beta_k)=(\mu_1(\alpha,\beta_1),\ldots,\mu_1(\alpha,\beta_k))
$$
where $\alpha$ and the $\beta_i$ represent pairs of complex numbers.
\end{proof}

\begin{corolario}\label{primercasointeresante}
If $n_1=1$ and $0<t\leq\infty$, then
$$
\mathbb{CP}_{\n}(t)\cong\mathbb{CP}_{\m}(t)\times\prod_{i\in J} S^{2n_i+1}.
$$
Here $J$ consists of all indices $i\in\{2,\ldots,r\}$ for which $n_i$ is odd, and $\m$ is the tuple obtained from $\n$ by removing all the $n_i$ with $i\in J$.
\end{corolario}

As in Example~\ref{paran10}, this immediately leads to additional examples of spaces $\L_{\n}(t)$ which are parallelizable. We will not bother to elaborate on this as the parallelizability property for any $\L_{\n}(t)$ will be characterized in Section~\ref{mnfdppts} in terms of a simple arithmetical condition.

\medskip
The matrices in~(\ref{tresmatrices}) are given in~\cite[Case $s=3$ in Subsection~2.5]{Eck} as one of the two (up to equivalence) irreducible solutions to the unitary Hurwitz-Radon ($2\times 2$)-matrix equations. However, the techniques in~\cite{Eck} do not seem to be usable for producing \invariant maps $S^{2a+1}\times S^{2b+1}\to S^{2b+1}$ with $a\geq2$. It is hoped that the above discussions encourage a systematic study  of (the existence and) the algebro-geometric properties of \invariant maps.

\bigskip
We next turn to the decomposition of complex-projective and lens product spaces after a suspension. The argument in our case can be streamlined since, unlike the situation for real projective product spaces in~\cite{DDa}, we are able to work directly with integral cohomology. In fact, the decomposition arises from a partial generalization (given by the first homotopy equivalence in Corollary~\ref{decomthom} below) of the well known fact that, for $\overline{n}=(n_1)$ and $k>0$, the Thom complex $\hspace{.5mm}T\left(k\gamma_{\overline{n}}(t)\right)$ is homeomorphic to a stunted space:
\begin{equation}\label{clastun}
T\left(k\gamma_{\overline{n}}(t)\right)\cong\L_{(n_1+k)}(t)\left/\,\L_{(k-1)}(t)\right..
\end{equation}

\begin{lema}\label{ThomDec}
Let $k>0$ and $\overline{n}=(n_1,\ldots,n_r)$ with $n_1\leq\cdots\leq n_r$. Let $|\overline{n}|=\sum n_i$. There is a map 
\begin{equation}\label{mapfnt}
f_{\overline{n},t}\hspace{.3mm}\colon\Sigma^{r-1}T\left(k\gamma_{\overline{n}}(t)\right)\to\L_{(|\overline{n}|+k+r-1)}(t)\left/\,\L_{(|\overline{n}|-n_1+k+r-2)}(t)\right.
\end{equation}
which in integral cohomology induces a monomorphism with image given by the $r-1$ suspension of $H^*(\L_{(n_1)}(t);\mathbb{Z})\cdot x_2 \cdots x_r\cdot U_{k\gamma_{\overline{n}}(t)}$.
\end{lema}
\begin{proof}
Use complex coordinates for unit spheres and discs $S^{2\ell-1}\subset D^{2\ell}\subset\mathbb{C}^\ell$ to define a map $g\colon [-1,1]^{r-1}\times S^{2n_1+1}\times\cdots\times S^{2n_r+1}\times D^{2k}\to D^{2(|\overline{n}|+k+r)}$ by
\begin{equation}\label{rollo}
g\left(t_2,\ldots, t_r,y_1,y_2,\ldots, y_r,x\right)=\left(\frac{M}{L}(t_2y_2,\ldots,t_{r}y_r,x), \sqrt{1-M^2}\,y_1\right)
\end{equation}
where $L=||(t_2y_2,\ldots,t_{r}y_r,x)||$ and $M=\max\left\{|t_2|,\ldots,|t_r|,||x||\right\}$. (Note that continuity forces the map to take the value $(0,\ldots,0,y_1)$ for $L=0$.) Since $||y_i||=1$, the right hand side of~(\ref{rollo}) lands in $S^{2(|\overline{n}|+k+r)-1}$, and in fact in $S^{2(|\overline{n}|-n_1+k+r-1)-1}$ when $M=1$. Further, $g$ is clearly equivariant with respect to the diagonal $S^1$-action (defined to be trivial on the cube $[-1,1]^{r-1}$). This yields a map $f_{\overline{n},t}$ as in~(\ref{mapfnt}) which, by direct verification, is a homeomorphim on the interior of the top cells. Therefore $f_{\overline{n},t}$ induces an isomorphism on the top dimension cohomology. Note that, for finite $t$, this top cell appears in one dimension higher than in the case $t=\infty$. Then the full behavior of $f_{\overline{n},t}$ in cohomology follows, for $t=\infty$, from the naturality of the definition of $g$  with respect to decreasing values of $n_1$ and, for finite $t$, from the naturality of $g$ with respect to the canonical projection $\pi$ in~(\ref{pullbackttprima}).
\end{proof}

A standard application of the Freudenthal suspension theorem shows that the space on the right hand side of~(\ref{mapfnt}), as well as  $f_{\overline{n},t}$ itself, desuspend $r-1$ times. We get a map
$$
\overline{f}_{\overline{n},t}\hspace{.3mm}\colon T\left(k\gamma_{\overline{n}}(t)\right)\to\Sigma^{1-r}\!\left(\L_{(|\overline{n}|+k+r-1)}(t)\left/\,\L_{(|\overline{n}|-n_1+k+r-2)}(t)\right.\right)
$$
between simply connected spaces such that $H^*(\overline{f}_{\overline{n},t};\mathbb{Z})$ is injective with $$\mbox{Im}(H^*(\overline{f}_{\overline{n},t};\mathbb{Z}))=H^*(\L_{(n_1)}(t);\mathbb{Z})\cdot x_2 \cdots x_r\cdot U_{k\gamma_{\overline{n}}(t)}.$$
In these conditions, for each subset $\sigma\subseteq\{2,\ldots,r\}$, set $r_\sigma=\mbox{card}(\sigma)+1$ and let $\overline{n}_\sigma$ be the $r_\sigma$-tuple obtained by removing from $\overline{n}$ those coordinates $n_i$ with $i\in\{2,\ldots,r\}-\sigma$. 

\begin{corolario}\label{decomthom}
There are homotopy equivalences
$$\hspace{.3mm}\Sigma\, T(k\gamma_{\overline{n}}(t))\simeq\bigvee_{\sigma\subseteq\{2,\ldots,r\}}\Sigma^{2-r_\sigma}\!\left(
\L_{(|\overline{n}_\sigma|+k+r_\sigma-1)}(t)\left/\,\L_{(|\overline{n}_\sigma|-n_1+k+r_\sigma-2)}(t)\right.\rule{-1.5mm}{4mm}\right)
$$
and, by setting $\L_{(-1)}(t)$ to be the base point,
$$\Sigma\L_{\overline{n}}(t)\simeq\bigvee_{\sigma\subseteq\{2,\ldots,r\}}\Sigma^{2-r_\sigma}\!\left(
\L_{(|\overline{n}_\sigma|+r_\sigma-1)}(t)\left/\,\L_{(|\overline{n}_\sigma|-n_1+r_\sigma-2)}(t)\right.\rule{-1.5mm}{4mm}\right).$$
\end{corolario}

In this result, the second homotopy equivalence can of course be interpreted as the case with $k=0$ in the first one.
\begin{proof}
By~(\ref{clastun}), we can assume $r\geq2$. For each subset $\sigma\subseteq\{2,\ldots,r\}$ (including the empty set), the projection $p_\sigma\colon \L_{\overline{n}}\to\L_{\overline{n}_\sigma}$  yields on Thom complexes a map $\pi_\sigma\colon T(k\gamma_{\overline{n}}(t))\to T(k\gamma_{\overline{n}_\sigma}(t))$. The observations right after the statement of Theorem~\ref{additive} imply that the composition
$
\overline{f}_{\overline{n},\sigma,t}=\overline{f}_{\overline{n}_\sigma,t}\circ\pi_\sigma
$
is injective in integral cohomology with image given by $$H^*(\L_{(n_1)}(t);\mathbb{Z})\cdot \prod_{i\in\sigma}x_i\cdot U_{k\gamma_{\overline{n}}(t)}.$$

The first equivalence now follows from a standard homotopy argument (based on Whitehead's Theorem) where the suspensions $\Sigma\overline{f}_{\overline{n},\sigma,t}$ are taken as the needed wedge components. The second equivalence is obtained in a similar way, except that we need an additional consideration. Namely, there is a homotopy equivalence
\begin{equation}\label{wedfre}
\Sigma\L_{\overline{n}}(t)\simeq T((n_r+1)\gamma_{\overline{m}}(t))\vee\Sigma\L_{\overline{m}}(t)
\end{equation}
coming from the fact that, in the cofiber sequence~(\ref{Cof}), the map $p$ is a retraction, so that $i$ is null homotopic (recall $\overline{m}=(n_1,\ldots,n_{r-1})$). We can then precompose the wedge component maps discussed in the previous paragraph for the Thom complex summand, as well as the (inductively generated) wedge component maps on $\Sigma\L_{\overline{m}}(t)$, with the corresponding wedge projections in~(\ref{wedfre}). This time we do not have to suspend the resulting wedge component maps, since we already have the co-H-structure of $\Sigma\L_{\overline{n}}(t)$ to add them up.
\end{proof}

\section{Category and topological complexity}
As a result of a well-understood numeric characterization for the existence of $\mathbb{Z}_2$-invariant maps, \cite[Theorem~2.20(1)]{DDa} describes maximal cartesian factors by products of spheres holding inside real projective product spaces $\mathrm{P}_{(n_1,n_2,\ldots,n_r)}$. This yields partial information about sub-additive invariants---such as the Lusternik-Schnirelmann category (cat) and the Farber topological complexity (TC)---of $\mathrm{P}_{(n_1,n_2,\ldots,n_r)}$. In turn, this suggests the possibility that, for real projective product spaces, the two homotopy invariants mentioned above can be estimated in terms of~$n_1$ and~$r$ alone. Such a guess is formalized in~\cite[Theorem~3.8]{gggx}---even if no sphere splits off $\mathrm{P}_{(n_1,n_2,\ldots,n_r)}$ as a cartesian factor. In this section we adapt the viewpoint in~\cite{gggx} in order to obtain similar estimates for the category and the topological complexity of complex-projective and lens product spaces.

\medskip
Start by observing that Varadarajan's relation in~\cite{vara} for the categories of the spaces involved in the Borel-type fibration
\begin{equation}\label{bortyp}
S^{2n_2+1}\times\cdots\times S^{2n_r+1}\longrightarrow\L_{\overline{n}}(t)\stackrel{p}\longrightarrow\L_{(n_1)}(t)
\end{equation}
gives
$$
\cat(\L_{\overline{n}}(t))\leq r(\cat(\L_{(n_1)}(t))+1)-1
$$
and, in particular,
\begin{equation}\label{estuno}
\TC(\L_{\overline{n}}(t))\leq2r(\cat(\L_{(n_1)}(t))+1)-2.
\end{equation}
The right hand side in~(\ref{estuno}) is explicit:
$$\cat(\L_{(n_1)}(t))=\begin{cases}2n_1+1,&t<\infty;\\n_1,&t=\infty\end{cases}$$
(see for instance~\cite[Remark~2.46]{clot}).

\medskip
The main result in this short section is:
\begin{theorem}\label{linref}
$
\TC(\L_{\overline{n}}(t))\leq r\left(1+\TC(\L_{(n_1)}(t))\right)-1.
$
\end{theorem}

Since $\TC\leq2\hspace{.3mm}\cat$, this improves on~(\ref{estuno}) by at least $r-1$. Slightly larger improvements would be expected for suitably small values of the $t$ (see~\cite{y1,y3,y2}).

\medskip
Theorem~\ref{linref} follows from \cite[Theorem~6.21]{cgstc} (where unreduced notation for Schwarz' genus is used),~\cite[Theorem~4.2]{gggx}, and the following extension of~\cite[Example 6.15]{cgstc}:

\begin{lema}\label{stcs}
For $1\leq t\leq\infty$, Colman-Grant's $\mathbb{Z}_t$-equivariant topological complexity of an odd sphere is $1$.
\end{lema}
\begin{proof}
Consider the system of local motion planners for $S^{2n+1}$ with domains
\begin{eqnarray*}
U_0&=&\{(A,B)\in S^{2n+1}\times S^{2n+1}\mid A\neq -B\},\\U_1&=&\{(A,B)\in S^{2n+1}\times S^{2n+1}\mid A\neq B\},
\end{eqnarray*}
and with local rules $s_0$ and $s_1$ given so that $s_0(A,B)$ is the shortest geodesic path from $A$ to $B$ (traveled at constant velocity), and $s_1(A,B)$ is obtained by concatenating
\begin{itemize}
\item the path from $A$ to $-A$ along the great circle in the direction determined by $v(A)$ and
\item  the path from $-A$ to $B$ along the shortest geodesic. 
\end{itemize}
Here $v\colon S^{2n+1}\to S^{2n+1}$ is the vector field given by complex multiplication by the imaginary unit $i$. Since $v$ is $S^1$-equivariant, it is readily checked that this is a $\mathbb{Z}_t$-invariant system of local motion planners. The result then follow since $1=\TC(S^{2n+1})\leq\TC_{\mathbb{Z}_t}(S^{2n+1})\leq1$.
\end{proof}

\section{Immersions, (stable) span, and (stable) parallelizability}\label{mnfdppts}
In what follows we set $\mathbb{C}_{\overline{n}}=\mathbb{C}^{n_1+1}\times\cdots\times\mathbb{C}^{n_r+1}=\mathbb{C}^{|\overline{n}|+r}$ and $S_{\overline{n}}=S^{2n_1+1}\times\cdots\times S^{2n_r+1}$ where, as usual, $\overline{n}=(n_1,\ldots,n_r)$ and $|\overline{n}|=\sum n_i$. The key result in this section is the identification of the stable isomorphism type of the tangent bundle of $\L_{\n}(t)$.

\begin{proposition}\label{tgtetodot}
For any $t\leq\infty$, the tangent bundle $\tau(\L_{\n}(t))$ is stably equivalent to $(|\n|+r)\gamma_{\n}(t)$.
\end{proposition}
\begin{proof}
The argument for $t=2$ in~\cite[(3.5)]{DDa} generalizes to any finite $t$: It is elementary to note that the tangent bundle $\tau(L_{\overline{n}}(t))$ is the quotient of 
$$
\left\{\rule{0mm}{3.5mm}(\overline{x},\overline{z})=(x_1,\ldots,x_r,z_1,\ldots,z_r)\in S_{\overline{n}}\times\mathbb{C}_{\overline{n}}\::\: x_i \perp z_i \:\:\forall i
\right\}
$$
by the identifications $(\overline{x},\overline{z})\sim\lambda (\overline{x},\overline{z})$ where $\overline{x}\in S_{\overline{n}}$, $\overline{z}\in\mathbb{C}_{\overline{n}}$, and $\lambda\in\mathbb{Z}_t$. Together with~(\ref{totalspace}) this yields an isomorphism
\begin{equation}\label{tgt}
\tau(L_{\overline{n}}(t))\oplus r\varepsilon_{\mathbb{R}}\stackrel{\cong}\longrightarrow\left(|\overline{n}|+r\right)\gamma_{\overline{n}}(t)
\end{equation}
induced by $(\overline{x},\overline{z},\overline{u})\mapsto(\overline{x},z_1+u_1x_1,\ldots,z_r+u_rx_r)$, where $\varepsilon_{\mathbb{F}}$ stands for a trivial $\mathbb{F}$-vector line bundle, and $\overline{u}=(u_1,\ldots,u_r)\in\mathbb{R}^r$. 

\smallskip
On the other hand, for $t=\infty$, we only need to note that the normal bundle to the embedding $S_{\n}\subset\mathbb{C}_{\n}$ is trivial, and that the diagonal $S^1$-action is compatible with the trivialization. Then, the desired conclusion follows directly from~(\ref{totalspace}) and~\cite[Lemma~3.1]{MR0353348}.
\end{proof}

In particular, the stable normal bundle of $\L_{\overline{n}}(t)$ is $-\left(|\overline{n}|+r\right)\gamma_{\n}(t)$ which, in view of~(\ref{pullbackttprimaprima}) and~(\ref{pullbackttprimaprimaprima}), has the same geometric dimension (dg) as that of $-\left(|\overline{n}|+r\right)\gamma_{(n_1)}(t)$. Consequently, Hirsch's characterization of the Euclidean immersion dimension (imm) of manifolds, yields the ($t\leq\infty$)-extension of \cite[Theorem~3.4]{DDa} in Corollary~\ref{immt}(\ref{5.2.1}) below. An analogous use of Proposition~\ref{tgtetodot}, (\ref{pullbackttprimaprima}), and~(\ref{pullbackttprimaprimaprima}) yields Corollary~\ref{immt}(\ref{5.2.2}).

\begin{corolario}\label{immt}
$\L_{\overline{n}}(t)$ is an orientable closed smooth manifold with \emph{
\begin{enumerate}[(i)]
\item\label{5.2.1}\emph{$\imm(\L_{\overline{n}}(t))=2|\n|+r-\delta_t+\max\left\{1,\gd\left(-(|\overline{n}|+r)\gamma_{(n_1)}(t)\right)\right\}$.}
\item\label{5.2.2} $\stablespan(\L_{\n}(t))=\span\left((|\n|+r)\gamma_{(n_1)}(t)\right)-r-\delta_t$.
\end{enumerate}}

\noindent Here $\delta_t=0$ for finite $t$, while $\delta_{\infty}=1$, so that $\dim(\L_{\n}(t))=2|\n|+r-\delta_t$.
\end{corolario}

As noted in~\cite[Corollary~3.7]{DDa}, Corollary~\ref{immt}(\ref{5.2.1}) readily yields instances of manifolds $\L_{\n}(t)$ immersing in Eucldiean spaces in such a way that the ratio $\mbox{codim}/\hspace{-.8mm}\dim$ is arbitrarily small. Likewise, the second item in Corollary~\ref{immt} yields instances of manifolds $\L_{\n}(t)$ for which the ratio $\stablespan/\hspace{-.8mm}\dim$ is arbitrarily close to 1.

\medskip
Corollary~\ref{immt} translates manifold properties for $\L_{\n}(t)$ in terms of the (largely unknown) geometric dimension of (stable) multiples of the Hopf bundles over standard complex projective and lens spaces. Yet, in what follows we take advantage of Corollary~\ref{immt} in order to give explicit characterizations for the parallelizability and stable parallelizability of the complex-projective and lens spaces---just as in the case of their standard counterparts.

\medskip
Corollary~\ref{immt}(\ref{5.2.2}) reduces the task of characterizing the stable parallelizability of $\L_{\n}(t)$ to that of determining the $\widetilde{\mathrm{KO}}$-order of the Hopf bundle over usual complex projective and lens spaces. In view of Example~\ref{paran10}, we only need to state the case for $n_1\geq1$. For a prime $p$ and a positive integer~$t$, let $\nu_p(t)$ denote the exponent in the largest $p$-power dividing $t$. Then define the integer $\sigma(n_1,t)$ so that $\nu_p(\sigma(n_1,t))=0$ unless $t$ is divisible by $p$, in which case $\nu_p(\sigma(n_1,t))$ is determined by the following formula, whose $j$-th instance applies only if all previous instances do not apply:
$$
\nu_p(\sigma(n_1,t))=\begin{cases}
n_1+1,&\mbox{if $p=2$, $\nu_2(t)=1$, and $n_1\not\equiv3\bmod4$;}\\
n_1,&\mbox{if $p=2$, and $\max\{\nu_2(t),n_1\}=1$;}\\
\nu_2(t)+n_1-1,&\mbox{if $p=2\,$ and $n_1\equiv0\bmod2$};\\
\nu_2(t)+n_1-2,&\mbox{if $p=2$};\\
\nu_p(t)+\left[\frac{n_1-2}{p-1}\right],&\mbox{if $n_1\geq2$};\\
$0$.&
\end{cases}
$$
\begin{corolario}\label{stabpara}
Assume $n_1\geq1$. The manifold $\L_{\n}(\infty)$ is stably parallelizable if and only if $n_1=1$ and $|\n|+r$ is even. For finite $t$, $\L_{\n}(t)$ is stably parallelizable if and only if $|\n|+r$ is divisible by $\sigma(n_1,t)$.
\end{corolario}
\begin{proof}
The case $t=\infty$ is straightforward: as shown in~\cite[Theorem~2.2(i)]{MR0171285}, the $\widetilde{\mathrm{KO}}$-order of the Hopf bundle over a standard complex projective space $\L^{n_1}$ is either 2 or infinite,
with the former case holding precisely for $n_1=1$. 

\smallskip
The required KO-information for a positive integer $t$ can be collected from work in the 70's. Let $t=\prod p_i^{e_i}$ be the primary decomposition of $t$. As shown in Theorem~1.6 of~\cite[Chapter~5]{MR592956}, the $p_i$-primary component of the $\widetilde{\mathrm{KO}}$-order of $\gamma_{(n_1)}(t)$ agrees with that of $\gamma_{(n_1)}(p_i^{e_i})$ (restricted to the $2n_1$-skeleton of its base, if $n_1\equiv0\bmod4$ and $p_i\neq2$, but this makes no difference since, in such a case, the two $\widetilde{\mathrm{KO}}$-groups differ only by a $\mathbb{Z}_2$ summand). There is a finite number of primes contributing since the $\widetilde{\mathrm{KO}}$-order of $\gamma_{(n_1)(t)}$ is relatively prime to any prime $p_i$ with $e_i=0$, for in fact $\gamma_{(n_1)(t)}$ pulls back trivially under the projection $S^{2n_1+1}=L^{2n_1+1}(1)\to L^{2n_1+1}(t)$. The conclusion follows since $p_i^{\nu_{p_i}(\sigma(n_1,t))}$ is the $\widetilde{\mathrm{KO}}$-order of $\gamma_{(n_1)}(p_i^{\nu_{p_i}(t)})$, as described explicitly in~\cite[Theorem~7.4]{AdamsVFS} for $p_i^{e_i}=2$, in~\cite[Theorem~1.4]{MR485765} for $p_i=2$ with $e_i\geq2$, and in~\cite[Theorem~1.1]{MR0312496} for $p_i>2$.
\end{proof}

The parallelizability of $\L_{\n}(t)$ can now be characterized explicitly (Examples~\ref{parability}) using the results in~\cite{MR0200937} on the parallelizability of stably parallelizable manifolds. With this in mind, we next recall one of the relevant invariants in Bredon-Kosinski's work.

\begin{definition}\label{kerinv}
For a smooth manifold $M$ of dimension $n$, the Kervaire semi-characteristic $\chi^*(M)$ is given in terms of the Euler characteristic $\chi(M)$ and the $\mathbb{F}_2$-Betti numbers $\beta_j(M)$ of $M$ as
$$
\chi^*(M)=\begin{cases}
\frac{\chi(M)}{2},&\mbox{if $n\equiv0\bmod2$;}\\
\left[\sum_{i\geq0}\beta_{2i}(M)\right],&\mbox{if $n\equiv1\bmod2$,}
\end{cases}
$$
where $\left[\ell\right]$ denotes the congruence class modulo $2$ of the integer $\ell$.
\end{definition}

Both Euler and Kervaire characteristics can be read off directly in the description in Theorem~\ref{additive} of the $\mathbb{F}_2$-cohomology of $\L_{\n}(t)$.

\begin{lema}\label{axlrs}
For any $t\leq\infty$,
\begin{enumerate}[(i)]
\item\label{chi0} $
\chi(\L_{\n}(t))=\begin{cases}n_1+1,&r=1\mbox{ and }t=\infty;\\0,&\mbox{otherwise.}\end{cases}
$
\item\label{chi1} If $r-\delta_t$ is odd (i.e.~if $\dim(\L_{\n}(t))$ is odd),
$$
\chi^*(\L_{\n}(t))=\begin{cases}
\left[n_1+1\right],&\mbox{if }\hspace{.2mm}(r=1\mbox{ and }t\mbox{ is even})\mbox{ or }(r\leq2\mbox{ and }t=\infty);
\\\left[1\right],&\mbox{if }\hspace{1.3mm}r=1\mbox{ and }t\mbox{ is odd};\\\left[0\right],&\mbox{otherwise.}\end{cases}
$$
\end{enumerate}
\end{lema}

\begin{examples}\label{parability}{\em 
The Poincar\'e-Hopf theorem and Lemma~\ref{axlrs}(\ref{chi0}) imply that $\L_{\n}(t)$ admits a nowhere vanishing vector field if and only if
\begin{equation}\label{nsc}
\mbox{$r>1\;$ or $\;t<\infty$.}
\end{equation}
On the opposite side of the picture, we use Corollary~\ref{stabpara}, Lemma~\ref{axlrs} and Bredon-Kosinski's results in~\cite{MR0200937} to give the promised characterization of the existence of the {\it largest} possible number of linearly independent fields on $\L_{\n}(t)$. Namely, for $n_1\geq1$ (recall Example~\ref{paran10}),
$$
\mbox{$\L_{\n}(\infty)$ is parallelizable if and only if $n_1=1$, $r>1$, and $|\n|+r$ is even.}
$$
Likewise, for $t<\infty$, $n_1\geq1$, and $r>1$ (recall, say from~\cite[Fact~6.2]{MR2559639}, that the parallelizability of standard lens spaces is well understood), a stably parallelizable $\L_{\n}(t)$ (as characterized in Corollary~\ref{stabpara}) must actually be parallelizable.
}\end{examples}

Lastly we remark that, as noted in~\cite[Theorem~3.9]{DDa} for $t=2$, Lemma~\ref{axlrs}, \cite[Theorems~20.1 and~20.10]{MR611333}, and the easy-to-check fact that $\L_{\n}(t)$ is a Spin manifold if and only if either $n_1=0$ or $|\n|+r$ is even, yield:

\begin{proposition}\label{invts}
Assume $r>1$ or $t<\infty$. The equality $\span=\stablespan$ holds for $\L_{\n}(t)$ under either one of the following two conditions:{\em
\begin{enumerate}
\item {\em $r-\delta_t$ is even.}
\item {\em $\dim(\L_{\n}(t))\equiv3\bmod8$, $\chi^*(\L_{\n}(t))=0$, and either $n_1=0$ or $|\n|+r$ is even.}
\end{enumerate}}

\noindent
If the first and third conditions in $(2)$ hold, but $\chi^*(\L_{\n}(t))\neq0$, then in fact $\span(\L_{\n}(t))=3$.
\end{proposition}

\bigskip\bigskip\small
\noindent{\sc Departamento de Matem\'aticas, Centro de Investigaci\'on y de Estudios Avanzados del IPN, M\'exico City 07000, M\'exico}

\noindent{\it E-mail address:} {\bf jesus@math.cinvestav.mx}

\bigskip\medskip

\noindent{\sc Colegio de Ciencia y Tecnolog\'ia, Universidad Aut\'onoma de la
Ciudad de M\'exico, Mexico City 09790, M\'exico}

\noindent{\it E-mail address}: {\bf maurilio.velasco.fuentes@uacm.edu.mx}

\end{document}